\documentclass{amsart}
\usepackage{amssymb}
\usepackage{amscd}
\usepackage{amssymb}
\usepackage{amsthm}

\input xy
\xyoption{all}

\newcommand{\supp}{\mathrm{supp}}

\newcommand{\pr}{\mathrm{pr}}

\newcommand{\id}{\mathrm{id}}

\newcommand{\M}{\mathcal M}
\newcommand{\K}{\mathcal K}

\newcommand{\N}{\mathbb N}
\newcommand{\R}{\mathbb R}

\newcommand{\Cl}{\mathrm{Cl}}

\newtheorem{theorem}{Theorem}[section]

\newtheorem{problem}[theorem]{Problem}
\newtheorem{lemma}[theorem]{Lemma}

\newtheorem{proposition}[theorem]{Proposition}
\theoremstyle{definition}
\newtheorem{df}[theorem]{Definition}

\title{Bundle of idempotent measures}

\author[T.~Radul]{Taras Radul}
\address[T.~Radul]{Kazimierz Wielki University, Bydgoszcz (Poland) and Ivan Franko National University of Lviv (Ukraine)}
\email{tarasradul@yahoo.co.uk}


\begin{document}
\begin{abstract} We investigate when the idempotent barycenter map restricted to the points with no-trivial fibers  is a trivial bundle with the fiber Hilbert cube.
\end{abstract}

\maketitle

\section{Introduction}The notion of idempotent (Maslov) measure finds important applications in different
parts of mathematics, mathematical physics and economics (see the survey article
\cite{Litv} and the bibliography therein). Topological and categorical properties of the functor of idempotent measures were studied in \cite{Zar}. Although idempotent measures are not additive and corresponding functionals are not linear, there are some parallels between topological properties of spaces of probability measures and spaces of idempotent measures (see for example \cite{Zar} and \cite{Radul}) which are based on existence of  natural equiconnectedness structure on both constructions.

However, some differences appear when the problem of the openness of the barycenter map was studying.
The problem of the openness of the barycenter map of probability measures was investigated  in \cite{Fed}, \cite{Fed1},  \cite{Eif}, \cite{OBr} and \cite{Pap}. In particular, it was proved in \cite{OBr} that the barycenter map for a compact convex set  in a locally convex space is open iff the map $(x, y)\mapsto 1/2 (x + y)$ is open and it is proved in \cite{Fed1} that the product of barycentrically open compact convex sets (i.e. compact convex sets for which the barycentre map  is open) is again barycentrically open. It was shown in  \cite{Radul1} that analogous statements for idempotent measures are false.

The problem of the softness of the barycenter map of probability measures was investigated  in \cite{Fed1}, \cite{Radul2} та \cite{Radul3}. In particular, it was proved in \cite{Fed1} that the product of $\omega_1$ barycentrically open metrizable convex compacta is barycentrically soft. Instead, it was proved in \cite{Radul4} that for the Tychonov cube of weight $\omega_1$ the idempotent barycenter map is not  soft.

It was proved in \cite{Fed} that the barycenter map of probability measures restricted to the points with no-trivial fibers  is a trivial bundle with the fiber Hilbert cube, provided it is open. However Zarichnyi remarked that this result can not be generalised for compacta with higher weights (see  \cite{Fed2} for more information).

We consider in this paper an analogous problem for idempotent measures.

\section{Idempotent measures: preliminaries}

By compactum we mean a compact Hausdorff space. We shall denote by $C(X)$ the
Banach space of continuous functions on $X$ endowed with the sup-norm. For any $c\in\R$ we shall denote  by $c_X$ the
constant function on $X$ taking the value $c$.

Let $\R_{\max}=\R\cup\{-\infty\}$ be the metric space endowed with the metric $\varrho$ defined by $\varrho(x, y) = |e^x-e^y|$.
Following the notation of idempotent mathematics (see e.g., \cite{MS}) we use the
notations $\oplus$ and $\odot$ in $\R$ as alternatives for $\max$ and $+$ respectively. The convention $-\infty\odot x=-\infty$ allows us to extend $\odot$ and $\oplus$  over $\R_{\max}$.

Max-plus convex sets were introduced in \cite{Z}.
Let $\tau$ be a cardinal number. Given $x, y \in \R^\tau$ and $\lambda\in\R_{\max}$, we denote by $y\oplus x$ the coordinatewise
maximum of x and y and by $\lambda\odot x$ the vector obtained from $x$ by adding $\lambda$ to each of its coordinates. A subset $A$ in $\R^\tau$ is said to be  max-plus convex if $\alpha\odot a\oplus  b\in A$ for all $a, b\in A$ and $\alpha\in\R_{\max}$ with $\alpha\le 0$. It is easy to check that $A$  is   max-plus convex iff $\oplus_{i=1}^n\lambda_i\odot\delta_{x_i}\in A$ for all $x_1,\dots, x_n\in A$ and $\lambda_1,\dots,\lambda_n\in\R_{\max}$ such that $\oplus_{i=1}^n\lambda_i=0$. In the following by max-plus convex compactum we mean a max-plus convex compact subset of $\R^\tau$.

We denote by $\odot:\R\times C(X)\to C(X)$ the map acting by $(\lambda,\varphi)\mapsto \lambda_X+\varphi$, and by $\oplus:C(X)\times C(X)\to C(X)$ the map acting by $(\psi,\varphi)\mapsto \max\{\psi,\varphi\}$.

\begin{df}\cite{Zar} A functional $\mu: C(X) \to \R$ is called an idempotent  measure (a Maslov measure) if

\begin{enumerate}
\item $\mu(1_X)=1$;
\item $\mu(\lambda\odot\varphi)=\lambda\odot\mu(\varphi)$ for each $\lambda\in\R$ and $\varphi\in C(X)$;
\item $\mu(\psi\oplus\varphi)=\mu(\psi)\oplus\mu(\varphi)$ for each $\psi$, $\varphi\in C(X)$.
\end{enumerate}

\end{df}

Let us remark that each idempotent measure $\mu: C(X) \to \R$ is a nonexpanding map (a Lipszicz map with constant $1$) with respect  to the sup-metric on the space  $C(X)$ and the natural metric on $\R$  \cite{Zar}.

Let $IX$ denote the set of all idempotent  measures on a compactum $X$. We consider
$IX$ as a subspace of $\R^{C(X)}$. It is shown in \cite{Zar} that $IX$ is a compact max-plus convex subset of $\R^{C(X)}$. The construction $I$ is  functorial what means that for each continuous map $f:X\to Y$ we can consider a continuous map $If:IX\to IY$ defined as follows $If(\mu)(\psi)=\mu(\psi\circ f)$ for $\mu\in IX$ and $\psi\in C(Y)$. It is proved in \cite{Zar} that the functor $I$ preserves topological embedding. For an embedding $i:A\to X$ we shall identify the space $F(A)$ and the subspace $F(i)(F(A))\subset F(X)$. Since the functor $I$ preserves intersections, we can define the notion of support of a measure $\mu\in IX$ as follows $\supp\mu=\cap\{A\subset X\mid A$ is closed and $\mu\in IA\}$ \cite{Zar}.

By $\delta_{x}$ we denote the Dirac measure supported by the point $x\in X$. We can consider a map $\delta X:X\to IX$ defined as $\delta X(x)=\delta_{x}$, $x\in X$. The map $\delta X$ is continuous, moreover it is an embedding \cite{Zar}. It is also shown in \cite{Zar} that the set $$I_\omega X=\{\oplus_{i=1}^n\lambda_i\odot\delta_{x_i}\mid\lambda_i\in\R_{\max},\ i\in\{1,\dots,n\},\ \oplus_{i=1}^n\lambda_i=0,\ x_i\in X,\ n\in\N\},$$ (i.e., the set of idempotent probability measures of finite support) is dense in $IX$.

The notion of density for an idempotent measure was introduced in \cite{A}. Let $\mu\in IX$. Then we can define a function $\rho_\mu:X\to [-\infty,0]$ by the formula $\rho_\mu(x)=\inf\{\mu(\varphi)|\varphi\in C(X)$ such that $\varphi\le 0$ and $\varphi(x)=0\}$, $x\in X$. The function $\rho_\mu$ is upper semicontinuous and is called the density of $\mu$. Conversely, each upper semicontinuous function $f:X\to [-\infty,0]$ with $\max f = 0$ determines an idempotent measure $\nu_f$
by the formula $\nu_f(\varphi) = \max\{f(x)\odot\varphi(x) | x \in X\}$, for $\varphi\in C(X)$ (existence of max follows from the upper semicontinuity of the function $f$). It is easy to check $\supp\mu=\Cl\{x\in X\mid \rho_\mu(x)>-\infty\}$.

Let $A\subset  \R^T$ be a compact max-plus convex subset. For each $t\in T$ we put $f_t=\pr_t|_A:A\to \R$ where $\pr_t:\R^T\to\R$ is the natural projection.    Given $\mu\in A$, the point $\beta_A(\mu)\in\R^T$ is defined by the conditions $\pr_t(\beta_A(\mu))=\mu(f_t)$ for each $t\in T$. It is shown in \cite{Zar} that $\beta_A(\mu)\in  A$ for each $\mu\in I(A)$ and the map $\beta_A : I(A)\to A$ is continuous.
The map $\beta_A$ is called the idempotent barycenter map.

A map $f:X\to Y$ between max-plus convex compacta $X$ and $Y$ is called max-plus affine if for each  $a, b\in X$ and $\alpha\in[-\infty,0]$ we have $f(\alpha\odot a\oplus  b)=\alpha\odot f(a)\oplus  f(b)$.  The map $b_X$  is max-plus affine for each max-plus convex compactum $X$ (Corollary 4.2 in \cite{Radul1}).

\section{A metric on the space of idempotent measures}

 Let $(X,d)$ be a metric compactum. By
$Lip_n(X,d)$ we denote the set of Lipschitz functions on $(X,d)$ with the
Lipschitz constant  $\le n$ for $n\in\N$. For $\nu$, $\mu\in IX$ and $n\in\N$ put $$d_n(\nu,\mu)=\sup\{\frac{|\nu(\varphi)-\mu(\varphi)|}{n}\mid \varphi\in Lip_n(X,d)\},$$
and
$$d_I(\nu,\mu)=\sum_{n\in\N}\frac{d_n(\nu,\mu)}{2^{n}}.$$

It is proved in \cite{Zar} that the metric $d_I$  generates the topology of $IX$, moreover the map   $\delta X:(X,d)\to (IX,d_I)$ is an isometric embedding.

For a metric space $(X,d)$ and $\varepsilon>0$ a map $h:X\to X$ is called $\varepsilon$-close to identity, if  $d(x,h(x))<\varepsilon$ for each $x\in X$.

\begin{lemma}\label{bt} Let $(X,d)$ be a metric compactum and a map $h:X\to X$ is $\varepsilon$-close to identity for some $\varepsilon>0$. Then the map $Ih:IX\to IX$ is $\varepsilon$-close to identity with respect to the metric $d_I$.
\end{lemma}

\begin{proof} Consider any $\mu\in IX$ and $\psi\in Lip_n(X,d)$. The lemma follows from inequalities $$\frac{|Ih(\mu)(\psi)-\mu(\psi)|}{n}\le \frac{|\mu(\psi\circ h)-\mu(\psi)|}{n}\le$$

(because $\mu$ is nonexpanding)

$$\le\frac{\sup_{x\in X}|\psi(h(x))-\psi(x)|}{n}\le\frac{\sup_{x\in X}nd(h(x),x)}{n}.$$
\end{proof}

\section{The main result for compacta $IX$}

In this section we investigate the map  $\beta_{IX}:I^2X\to IX$ for a metrizable compactum  $X$. It was proved in \cite{Radul1} that $\beta_{IX}:I^2X\to IX$ is open for each  compactum  $X$.

For a function $\varphi\in C(X)$ by $\tilde\varphi\in C(IX)$ we denote the function defined by the formula $\tilde\varphi(\nu)=\nu(\varphi)$ for $\nu\in IX$. Remark that the function  $\tilde\varphi\in C(IX)$ is affine. Diagonal product $(\tilde\varphi)_{\varphi\in C(X)}$ embeds $IX$ into $\R^{C(X)}$ as a Max-Plus convex subset. It is easy to  see that the map $\beta_{IX}$ satisfies the equality $\beta_{IX}(\M)(\varphi)=\M(\tilde\varphi)$ for any $\M\in I^2X=I(IX)$ and $\varphi\in C(X)$. Particularly we have $\beta_{IX}\circ I(\delta X)=\id_{IX}$ for each compactum $X$.

Let $X$ be a max-plus convex compactum. Following \cite{GK}  we call a point $x\in X$ an extremal point if for each two points $y,z\in X$ and for each $t\in [-\infty,0]$ the equality $x=t\odot y\oplus z$  implies $x\in\{y,z\}$. The set of extremal points of a max-plus convex compactum $X$ we denote by $E(X)$. It is proved in \cite{Radul1} that  the set $E(X)$ is closed if   the idempotent barycenter map is open.

\begin{lemma}\label{be} We have $E(IX)=\delta X(X)=\{\mu\in IX\mid |\beta_{IX}^{-1}(\mu)|=1\}$ for each compactum $X$.
\end{lemma}

\begin{proof} Consider any $\mu\in E(IX)$. Assume that there exist distinct  $x_1$, $x_2\in X$ such that  $\rho_\mu(x_1)>-\infty$ and $\rho_\mu(x_2)>-\infty$, where $\rho_\mu$ is the density of $\mu$. We can assume $\rho_\mu(x_1)=0$. Represent $X=X_1\cup X_2$, where $X_1$ and $X_2$ are closed subsets of  $X$, such that $x_2\notin X_1\ni x_1$ and $x_1\notin X_2\ni x_2$. Put $a=\max _{x\in X_2}\rho_\mu(x)\le 0$ (existence of max follows from the upper semicontinuity of the function  $\rho_\mu$). Define functions $\rho_1, \rho_2:X\to [-\infty,0]$ as follows

$$\rho_1(x)=\begin{cases}
\rho_\mu(x),&x\in X_1,\\
-\infty,&x\notin X_1\end{cases}$$
and
$$\rho_2(x)=\begin{cases}
\rho_\mu(x)-a,&x\in X_2,\\
-\infty,&x\notin X_2.\end{cases}$$

It is easy to see that the functions $\rho_1$ ≥ $\rho_2$ are upper semicontinuous and $\max _{x\in X}\rho_1(x)=\max _{x\in X}\rho_2(x)=0$. Hence they are densities of some $\mu_1$, $\mu_2\in IX$, moreover $\mu\notin\{\mu_1;\mu_2\}$ and $a\odot\mu_2\oplus\mu_1=\mu$. We obtained a contradiction and $\mu\in \delta X(X)$.

Consider any $x\in X$. Then $\delta_x=I(\{x\})$. Lemma 2.2 from \cite{Radul4} implies $\beta_{IX}^{-1}(I(\{x\}))\subset I^2(\{x\})$. Hence $\beta_{IX}^{-1}(\delta_x)=\beta_{IX}^{-1}(I(\{x\}))\subset I^2(\{x\})=\{\delta_{\delta_x}\}$ and $\delta_x\in\{\mu\in IX\mid |\beta_{IX}^{-1}(\mu)|=1\}$.

Finally, consider any $\mu\notin E(IX)$. There exist $\mu_1$, $\mu_2\in IX$ and $a\in [-\infty,0]$, such that $\mu\notin\{\mu_1;\mu_2\}$ and $a\odot\mu_2\oplus\mu_1=\mu$. Then the set $\beta_{IX}^{-1}(\mu)$ contains two-points subset $\{a\odot\delta_{\mu_2}\oplus\delta_{\mu_1},\delta_\mu\}$ ≥ $\mu\notin\{\nu\in IX\mid |\beta_{IX}^{-1}(\nu)|=1\}$.
\end{proof}

Let us recall some necessary notions and facts. Let  $A$ be a subset of a space $X$. A continuous map $r:X\to A$ is called a retraction, if $r|_A=\id_A$. A space  $A$ is called an absolute retract (shortly AR), if for each closed embedding $j:A\to X$ there exists a retraction $r:X\to j(A)$.

\begin{lemma}\label{ar} Each metrizable max-plus convex compactum $X$ is an absolute retract.
\end{lemma}

\begin{proof} It is known that $IX$ is an absolute retract for each metrizable compactum (see for example \cite{Radul4} or \cite{Zar1}). The map $\delta_X\circ\beta_X : IX\to \delta_X(X)$ is a retraction, hence $X$ is an absolute retract.
\end{proof}

A continuous map $f:X\to Y$ is said to be soft \cite{Shchep1} if for any  paracompact space $Z$, any closed subspace $A$ of $Z$ and
maps $\Phi:A\to X$ and $\Psi:Z\to Y$ with $\Psi|A=f\circ\Phi$ there exists a continuous
map $G:Z\to X$ such that $G|A=\Phi$ and $\Psi=f\circ G$. Ii follows from the definition that a restriction of a soft map to a complete preimage is soft.

A continuous map $f:X\to Y$ is  called a trivial bundle with a fiber  $Z$ (or trivial  $Z$-bundle), where $Z$ is a topological space, if there exists a homeomorphism $h:X\to Z\times Y$ such that $f=\pr_Y\circ h$. A continuous map $f:X\to Y$ is  called a locally trivial bundle with a fiber  $Z$ if for each point $y\in Y$ there is a neighborhood $U$ and a homeomorphism $h:f^{-1}(U)\to Z\times U$ such that $f(x)=\pr_Y\circ h(x)$ $x\in X$.

By  $Q$ we denote the Hilbert cube, that is the countable infinite product of unit closed intervals with the product topology. We will use the following version of Torunczyk-West Theorem about topological characterization of  $Q$-bundle, which follows immediately from Proposition 1.6.5 \cite{FedChig}.

\begin{theorem}\label{TW} Let $f:X\to Y$  be a soft map between compact AR-compacta, and $d$ is a metric on $X$. If for each $\varepsilon>0$ there exist two $\varepsilon$-close to identity continuous maps   $h,g:X\to X$ such that $h(X)\cap g(X)=\emptyset$ and $f\circ h=f=f\circ g$, then $f$ is a trivial $Q$-bundle.
\end{theorem}

Put  $J=\{(t,p)\in [-\infty,0]\times [-\infty,0]\mid t\oplus p=0\}$.
Let  $X$ be a   max-plus convex compactum. Define the map $s_X:X\times X\times J\to X$ by the formula $s_X(x,y,t,p)=t\odot x\oplus p\odot y$. It is easy to check that $s_X$ is continuous.

\begin{theorem}\label{MR} Let  $X$  be a metric compactum. The map $b_X=\beta_{IX}|_{I^2X\setminus\beta_{IX}^{-1}(E(IX))} :I^2X\setminus\beta_{IX}^{-1}(E(IX))\to IX\setminus E(IX)$ is a trivial $Q$-bundle.
\end{theorem}

\begin{proof} Let  $(X,d)$  be a metric compactum. We consider on the spaces $IX$ and $I^2X$ the metrics $d_I$ та $(d_I)_I$ introduced in Section 3.  Since each locally trivial $Q$-bundle is trivial \cite{TW}, it is enough to prove that the map $b_X$ is a locally trivial $Q$-bundle. Consider any point $\nu\in IX\setminus E(IX)$. Choose a  max-plus convex compact neighborhood  $K$ of the point $\nu$ such that $K\cap E(IX)=\emptyset$. It follows from the affinity of the map $\beta_{IX}$ that the compactum $\beta_{IX}^{-1}(K)$ is max-plus convex. Hence the compacta  $\beta_X^{-1}(K)$ and $K$ are absolute retracts according to Lemma \ref{ar}.

By  $b_K$ we denote the restriction of the map $\beta_{IX}$ onto $\beta_{IX}^{-1}(K)$.  It is enough to prove that the map $b_K$ satisfies the condition of Theorem \ref{TW}. Since the map $\beta_{IX} : I^2X\to IX$ is open and   $X$  is a metric compactum,  $\beta_{IX}$ is soft (Corollary 4.2 in  \cite{Radul4}). Then the map $b_K$ is soft too being a restriction of  $\beta_{IX}$ to complete preimage.

Consider any $\varepsilon>0$.   Since the map $s_{IX}$ is continuous and $s_{IX}(\mu,\nu,0,-\infty)=\mu$ for each $\mu$, $\nu\in IX$,  there is $t\in (-\infty,0]$ such that for each $\kappa\in K$ the continuous map  $h_\kappa:IX\to IX$ defined by the formula $h_\kappa(\mu)=s_{IX}(\mu,\kappa,0,t)$ is $\varepsilon$-close to identity. Let us remark that the set  $S=\cup_{\kappa\in K}h_\kappa(IX)=s_{IX}(IX,K,0,t)$ is compact and $S\cap E(IX)=\emptyset$.

Define the map $g:\beta_{IX}^{-1}(K)\to I^2X$ by the formula $g(\M)=I(h_{\beta_{IX}(\M)})(\M)$. It is easy to check that $g$ is continuous and Lemma \ref{bt} implies that $g$ is $\varepsilon$-close to identity.

Consider any function $\varphi\in C(X)$ and $\M\in\beta_{IX}^{-1}(K)$. Then we have $\beta_{IX}\circ g(\M)(\varphi)=\beta_{IX}\circ I(h_{\beta_{IX}(\M)})(\M)(\varphi)=\M(\tilde\varphi \circ h_{\beta_{IX}(\M)}).$
Denote $\Psi=\tilde\varphi \circ h_{\beta_{IX}(\M)}$. Choose any $\nu\in IX$. Then we have $\Psi(\nu)=\tilde\varphi \circ h_{\beta_{IX}(\M)}(\nu)=\tilde\varphi(t\odot \beta_{IX}(\M)\oplus \nu)=$ (since $\tilde\varphi:IX\to\R$ is affine) $=t\odot \M(\tilde\varphi)\oplus \nu(\varphi)\le\M(\tilde\varphi)\oplus \nu(\varphi)$. Thus  $\tilde\varphi\le\Psi\le\M(\tilde\varphi)_{IX}\oplus\tilde\varphi$. The properties of idempotent measure imply $\M(\tilde\varphi)=\M(\Psi)$ for each $\varphi\in C(X)$ and $\M\in\beta_{IX}^{-1}(K)$, hence  $\beta_{IX}\circ g=\beta_{IX}$. Particularly we have $g(\beta_{IX}^{-1}(K))\subset K$.
Additionally we have $\supp(g(\M))\cap E(IX)=\emptyset$ for each $\M\in\beta_{IX}^{-1}(K)$.

Since the map $s_{I^2X}$ is continuous and $s_{I^2X}(\M,\K,0,-\infty)=\M$ for each $\M$, $\K\in I^2X$,  there is $t\in (-\infty,0]$, such that the continuous map $h:\beta_{IX}^{-1}(K)\to \beta_{IX}^{-1}(K)$ defined by the formula $h(\M)=s_{I^2X}(\M,I(\delta X)(\beta_{IX}(\M)),0,t)$ is $\varepsilon$-close to identity. Since the map $\beta_{IX}$ is affine and $\beta_{IX}\circ I(\delta X)=\id_{IX}$, we have $\beta_{IX}\circ h=\beta_{IX}$. Since $\delta X(X)=E(IX)$, we have $I(\delta X)(\beta_{IX}(\M))\in I(E(IX))$. Hence $\supp(h(\M))\cap E(IX)\neq\emptyset$ for each $\M\in\beta_{IX}^{-1}(K)$. We obtained that the maps $g$ and $h$ satisfy the conditions of Theorem \ref{TW}.
\end{proof}

\section{The general case: a partial result and a general problem}

In this section we investigate topological structure of the idempotent barycenter map  for arbitrary metric  max-plus convex compactum. Let us remark that for probability measures  the set of extremal points coincides with the set of points with one-point fibers of the barycenter map \cite{Fed}. The situation generally is different in the case of idempotent measures. It is easy to see that the map  $\beta_{[0,1]}: I([0,1])\to [0,1]$ has no points with one-point fibers although the max-plus convex compactum $[0,1]$ has two extremal points $0$ and $1$.

For a max-plus convex compactum $X$ denote by $B(X)$ the set of points with one-point fibers of the map $\beta_X$.

\begin{lemma}\label{bb} Let $X$ be  a max-plus convex compactum.  Then $B(X)=\{x\in X\mid$ for each two points $y,z\in X$ and for each $\lambda\in (-\infty,0]$ the equality $x=\lambda\odot y\oplus z$ implies $x=y=z\}$.
\end{lemma}

\begin{proof} Consider any $x\in B(X)$. Assuming that there exist  $t\in (-\infty,0]$ and not equal simultaneously to   $x$  points
 $y$, $z\in X$ not equal simultaneously to   $x$ such that $x=\lambda\odot y\oplus z$, we obtain $\beta_X(\lambda\odot \delta_y\oplus \delta_z)=x$ and $\lambda\odot \delta_y\oplus \delta_z\neq \delta_x$ what contradicts to $x\in B(X)$.

 Now consider $x\notin B(X)$. Since $\delta_x\in \beta_{X}^{-1}(x)$, there is $\mu\in \beta_{X}^{-1}(x)$ and distinct points $x_1$, $x_2\in X$ such that $\rho_\mu(x_1)>-\infty$ ≥ $\rho_\mu(x_2)>-\infty$ where $\rho_\mu$ is the density of $\mu$. We can assume  $\rho_\mu(x_1)=0$.

 Recall that $X$ we consider as max-plus convex compact subset of $\R^T$. There is $t\in T$ such that $\pr_t(x_1)\neq\pr_t(x_2)$. We can assume $\pr_t(x_1)<\pr_t(x_2)$. Choose $q\in (\pr_t(x_1),\pr_t(x_2))$, such that $\pr_t(x)\neq q$.  Represent $X=X_1\cup X_2$ where $X_1=\{y\in X\mid\pr_t(y)\le q\}$ and $X_2=\{y\in X\mid\pr_t(y)\ge q\}$ are closed max-plus convex subsets of the compactum $X$. Put $a=\max _{y\in X_2}\rho_\mu(y)\le 0$ (existence of max follows from upper semicontinuity of the function  $\rho_\mu$). Define functions $\rho_1, \rho_2:X\to [-\infty,0]$ as follows

$$\rho_1(y)=\begin{cases}
\rho_\mu(y),&y\in X_1,\\
-\infty,&y\notin X_1\end{cases}$$
and
$$\rho_2(y)=\begin{cases}
\rho_\mu(y)-a,&y\in X_2,\\
-\infty,&y\notin X_2.\end{cases}$$

It is easy to see that the functions   $\rho_1$ and $\rho_2$ are upper semicontinuous and $\max _{y\in X}\rho_1(y)=\max _{y\in X}\rho_2(y)=0$. Hence they are densities of some $\mu_1$, $\mu_2\in IX$, moreover  $a\odot\mu_2\oplus\mu_1=\mu$. Put $\beta_{X}(\mu_1)=z$ та $\beta_{X}(\mu_2)=y$. Since  $X_1$ and $X_2$ are  max-plus convex compacta,
$z\in X_1$ and $y\in X_2$.  Hence the points
 $y$ and $z$  are not equal simultaneously to  $x$. We also have $x=\beta_{X}(\mu)=\beta_{X}(a\odot\mu_2\oplus\mu_1)=a\odot y\oplus z$.
\end{proof}

Let us remark that Lemma \ref{MR} implies the inclusion $B(X)\subset E(X)$.

\begin{theorem}\label{I} The map $\beta_{[0,1]}$ is trivial $Q$-bundle.
\end{theorem}

\begin{proof} We consider the natural metric $d$ on  $[0,1]$ and the metric $d_I$ introduced in Section 3 on $I([0,1])$.

It is enough to prove that the map  $\beta_{[0,1]}$ is locally trivial  $Q$-bundle. We denote by $b_0$ the restriction of the map $\beta_{[0,1]}$ to $\beta_{[0,1]}^{-1}([0,\frac{2}{3}])$ and by  $b_1$ we denote  the restriction of the map $\beta_{[0,1]}$ to $\beta_{[0,1]}^{-1}([\frac{1}{3},1])$. We will prove that the maps $b_0$ and  $b_1$ are trivial $Q$-bundle.  Since the map  $\beta_{IX}$ is affine, Lemma \ref{ar} implies that the compacta $\beta_{[0,1]}^{-1}([\frac{1}{3},1])$ and $\beta_{[0,1]}^{-1}([0,\frac{2}{3}])$ are absolute retracts.

Since the map  $\beta_{[0,1]}$ is open \cite{Radul1},  $\beta_{[0,1]}$ is soft (Corollary 4.2 in  \cite{Radul4}). The the maps $b_0$ and $b_1$ are soft too  being restrictions of $\beta_{[0,1]}$ to complete preimages.

 Consider any $\varepsilon>0$.   Since the map $s_{I([0,1])}$ is continuous and $s_{I([0,1])}(\mu,\nu,0,-\infty)=\mu$ for each  $\mu$, $\nu\in IX$,  there is $t\in (-\infty,-1]$ such that the continuous map $h_0:\beta_{[0,1]}^{-1}([0,\frac{2}{3}])\to \beta_{[0,1]}^{-1}([0,\frac{2}{3}])$ defined by the formula $h_0(\mu)=s_{IX}(\mu,\delta_1,0,t)$ is $\varepsilon$-close to identity. Since $t\leq-1$ we have $b_0\circ h_0=b_0$.  We have as well  $1\in\supp h_0(\mu)$ for each $\mu\in \beta_{[0,1]}^{-1}([0,\frac{2}{3}])$.

 We are going to define the function   $l:\beta_{[0,1]}^{-1}([0,\frac{2}{3}])\times [\frac{2}{3},1]\to \beta_{[0,1]}^{-1}([0,\frac{2}{3}])$ as follows. Consider any $\nu\in \beta_{[0,1]}^{-1}([0,\frac{2}{3}])$ with the density $\rho_\nu$ and $p\in [\frac{2}{3},1]$. We define the function $\sigma:[0,1]\to[-\infty,0]$ by the formula
 $$\sigma(t)=\begin{cases}
\rho_\nu(t),&t<p,\\
\max_{s\in[p,1]}(1-s)\odot\rho_\nu(s),&t=p,\\
-\infty,&t>p.\end{cases}$$
It is easy to see that the function $\sigma$ is well defined, upper continuous and   $\max_{s\in[0,1]}\sigma(s)=1$. Hence  $\sigma$ is a density for some $\mu\in \beta_{[0,1]}^{-1}([0,\frac{2}{3}])$, moreover $\beta_{[0,1]}(\mu)=\beta_{[0,1]}(\nu)$. Put $l(\nu,p)=\mu$. It is easy to check that the map  $l$ is continuous and $l(\nu,1)=\nu$ for each $\nu\in \beta_{[0,1]}^{-1}([0,\frac{2}{3}])$. Then there exists $\delta\in(\frac{2}{3},1)$ such that the continuous map $g_0:\beta_{[0,1]}^{-1}([0,\frac{2}{3}])\to \beta_{[0,1]}^{-1}([0,\frac{2}{3}])$ defined by the formula $g_0(\nu)=l(\nu,\delta)$ is $\varepsilon$-close to identyty. It is easy to see that  $b_0\circ g_0=b_0$.  We also have  $1\notin\supp g_0(\nu)$ for each $\nu\in \beta_{[0,1]}^{-1}([0,\frac{2}{3}])$.

Thus it follows from Theorem \ref{TW} that the map $b_0$ is trivial $Q$-bundle. Analogous, even simpler, arguments we use for the map $b_1$.
\end{proof}

A positive answer to the next question could be a good generalization of Theorems \ref{MR} and \ref{I}.

\begin{problem}\label{pr} Let  $X$ be a max-plus convex metric compactum, such that the map $\beta_X$ is open. Is the map $\beta_{X}|_{IX\setminus\beta_{X}^{-1}(B(X))} :IX\setminus\beta_{X}^{-1}(B(X))\to X\setminus B(X)$ a trivial $Q$-bundle?
\end{problem}

Let us remark that Theorems \ref{MR} and \ref{I} deal with two polar opposite cases: all extreme points are the points with one-point fibers in Theorem \ref{MR}, instead there are no points with one-point fibers in Theorem \ref{I}. The ideas of the proof of these theorems are rather different. Generally, a  max-plus convex compactum can have extremal points with as well trivial fibers as non-trivial fibers. So, it seems that for proof of a general theorem we should mix the methods of the proofs of both  theorems.

Finally we consider the problem whether it is possible to obtain analogous results for non-metrizable compacta. We consider compacta with the first uncountable weight  $\omega_1$. We should consider bundles with fibers  $[0,1]^{\omega_1}$ (let us remark that  $IX$ is homeomorphic to  $[0,1]^{\omega_1}$ for each openly generated homogeneous with respect by character compactum $X$ \cite{Zar1}). The analogical problem was considered for probability measures (see \cite{Fed2} and \cite{RZ}, where some negative results were obtained). We will see that we have the analogous situation for idempotent measures.

 Consider a max-plus convex compactum $I([0,1]^{\omega_1})$ and the corresponding idempotent barycenter map $\beta_{I([0,1]^{\omega_1})}$.  We have  $B(([0,1]^{\omega_1})=I\delta[0,1]^{\omega_1}([0,1]^{\omega_1})$ by Lemma \ref{be}. It follows from Lemma 2.2 in \cite{Radul4} that the fiber $\beta_{I([0,1]^{\omega_1})}^{-1}(\nu)$ is metrizable for each $\nu\in IA$, where $A$ is any metrizable subset of  $[0,1]^{\omega_1}$. Hence we can only consider an idempotent analogue of the Fedorchuk question stated for probability measures (Question 7.12 in \cite{Fed2}): Is the map $\beta_{I([0,1]^{\omega_1})}$ restricted to complete preimage of some compactum consisting of measures with non-metrizable supports a trivial bundle with the fiber $[0,1]^{\omega_1}$?

The answer to this question is negative and follows from the next proposition which is an idempotent analogue of Proposition 5.1 in \cite{RZ}. Moreover, the proof of our proposition can be obtained by a formal translation of   Proposition 5.1 to the language of the idempotent mathematics with taking into account some categorical properties of the functor of idempotent measures  $I$ investigated in \cite{Zar}.

\begin{proposition} There is a measure with non-metrizable support $\nu\in I([0,1]^{\omega_1})$ such that $\beta_{I([0,1]^{\omega_1})}^{-1}(\nu)$ has a point of countable character.
\end{proposition}

\vspace*{\parindent}


\begin{thebibliography}{100}

\bibitem{A} M.Akian, {\em Densities of idempotent measures and large deviations}, Trans. of Amer.Math.Soc. {\bf 351} (1999), no. 11, 4515Ц-4543.

\bibitem{Zar1} L.Bazylevych,D.Repovs,M.Zarichnyi, {\em Spaces of idempotent measures of compact metric spaces}, Topology Appl., {\bf 157}
(2010), 136--144.

\bibitem{Eif} L.Q. Eifler, {\em Openness of convex averaging}, Glasnik Mat. Ser. III, {\bf 32} (1977), no. 1, 67Ц-72.


\bibitem{Fed} V.V.~Fedorchuk, {\em  On a barycentric map of probability measures}, Vestn. Mosk. Univ, Ser. I, No 1, (1992), 42--47.

\bibitem{Fed1} V.V.~Fedorchuk, {\em On barycentrically open bicompacta}, Siberian Mathematical Journal, {\bf 33}
(1992), 1135--1139.

\bibitem{Fed2} V.V.~Fedorchuk, {\em Probability measures in topology}, Russian Mathematical Surveys, {\bf 46 } (1991), 41--80.

\bibitem{FedChig} {\it V.~Fedorchuk, A.~Chigogidze} Absolute retracts and infinitedimensional manifolds,  M.:Nauka, 1992, 231 p.

\bibitem{GK} S. Gaubert, R. Katz {\em Max-plus convex geometry}, Relations and Kleene algebra in computer science, 192Ц206, Lecture Notes in Comput. Sci., 4136, Springer, Berlin, 2006.

\bibitem{Litv} G. L.Litvinov, {\em The Maslov dequantization, idempotent and tropical mathematics: a very brief
introduction}, Idempotent mathematics and mathematical physics, 1Ц17, Contemp. Math., 377, Amer. Math. Soc., Providence, RI, 2005.

\bibitem{MS}   V.P. Maslov, S.N. Samborskii, {\em Idempotent Analysis}, Adv. Soviet Math., vol. 13, Amer. Math. Soc., Providence, 1992.

\bibitem{OBr} R.C. O'Brien, {\em On the openness of the barycentre map}, Math. Ann., {\bf 223} (1976), 207--212.

\bibitem{Pap} S. Papadopoulou, {\em  On the geometry of stable compact convex sets}, Math. Ann., {\bf 229} (1977), 193--200.

\bibitem{Radul} T.~Radul, {\em Absolute retracts and equiconnected monads}, Topology Appl. {\bf 202} (2016), 1--6.

\bibitem{Radul1} T.~Radul, {\em On the openness of the idempotent barycenter map}, Topology Appl. (2019), DOI 10.1016/j.topol.2019.07.003

\bibitem{Radul2} T.~Radul, {\em On the baricentric map of probability measures}, Vestn. Mosk. Univ., Ser. I  (1994), No.1, 3--6.

\bibitem{Radul3} T.~Radul, {\em On baricentrically soft compacta}, Fund.Math. {\bf 148} (1995), 27--33.

\bibitem{Radul4} T.~Radul, {\em Idempotent measures:absolute retracts and soft maps}, Topological Methods in Nonlinear Analysis (submitted).

\bibitem{RZ} T.N.Radul, M.M.Zarichnyi, {\em Monads in the category of
compacta}, Uspekhi Mat.Nauk. {\bf 50} (1995) 83-Ц108.

\bibitem{Shchep1} E.V.Shchepin, {\em Topology of limit spaces of uncountable inverse spectra}, Russian Mathematical Surveys {\bf 31} (1976), 155--191.

\bibitem{TW} H. Torunczyk, J. West, {\em Fibrations and Bundles with Hilbert Cube Manifold Fibers}, Memoirs of the American Mathematical Society {\bf 80} (1989), no.406.

\bibitem{vV}   M.van de Vel, {\em Theory of convex strutures}, North-Holland, 1993.


\bibitem{Zar} M.~Zarichnyi, {\em Spaces and mappings of idempotent measures}, Izv. Ross. Akad. Nauk Ser. Mat. {\bf 74} (2010), 45--64.

\bibitem{Zar2} M.~Zarichnyi, {\em Michael selection theorem for max-plus compact convex sets}, Topology Proceedings. {\bf 31} (2007), 677--681.

\bibitem{Z} K. Zimmermann, {\em A general separation theorem in extremal algebras}, Ekon.-Mat. Obz. {\bf 13} (1977) 179--201.

\end{thebibliography}
\end{document}